\documentclass[11pt,reqno]{amsart}
\usepackage{amssymb}
\usepackage[all]{xy}
\usepackage{fancyhdr}
\usepackage{stmaryrd}
\usepackage{hyperref}
\usepackage{tikz-cd}
\usepackage{cite}
\usepackage{amsmath,amsfonts}
\usepackage{graphicx}
\usepackage{wrapfig}
\usepackage{array}
\usepackage{amsthm}
\usepackage[left=1.2in, right=1.2in, top=1in, bottom=1in]{geometry}
\usepackage{etoolbox}
\patchcmd{\section}{\scshape}{\bfseries}{}{}
\makeatletter
\renewcommand{\@secnumfont}{\bfseries}
\makeatother
\usetikzlibrary{matrix,arrows,decorations.pathmorphing}
\usetikzlibrary{arrows,positioning}   
\input xy
\xyoption{all}
\usepackage{secdot}  

\theoremstyle{plain}

\DeclareMathOperator{\Ker}{Ker}

\DeclareMathOperator{\Aut}{Aut}
\DeclareMathOperator{\Inn}{Inn}

\newcommand{\hypergroup}{\mathbf{Hypgrp}}
\newcommand{\hypergroupfinite}{\mathbf{Hypgrp_{finite}}}

\newcommand{\asschmfinite}{\mathbf{Asschm_{finite}}}

\input xy
\xyoption{all}
\patchcmd{\abstract}{\scshape\abstractname}{\normalsize{\textbf{\abstractname}}}{}{}
\begin{document}

\title{Association schemes and Hypergroups}
\author{Jaiung Jun}
\address{Department of Mathematical Sciences, Binghamton University, Binghamton, NY 13902, USA}
\curraddr{}
\email{jjun@math.binghamton.edu}



\date{}

\dedicatory{}

\begin{abstract}
\normalsize{\noindent 
In this paper, we investigate hypergroups which arise from association schemes in a canonical way; this class of hypergroups is called realizable. We first study basic algebraic properties of realizable hypergroups. Then we prove that two interesting classes of hypergroups (partition hypergroups and linearly ordered hypergroups) are realizable. Along the way, we prove that a certain class of projective geometries is equipped with a canonical association scheme structure which allows us to link three objects; association schemes, hypergroups, and projective geometries (see, \S \ref{motivation} for details). 
} 
\end{abstract}

\maketitle


\theoremstyle{definition}
\newtheorem{mydef}{\textbf{Definition}}[section]
\newtheorem{myeg}[mydef]{\textbf{Example}}
\newtheorem{rmk}[mydef]{\textbf{Remark}}
\newtheorem{defnot}[mydef]{\textbf{Definition/Notation}}
\newtheorem*{que}{\textbf{Question}}
\newtheorem*{normk}{\textbf{Remark}}

\theoremstyle{plain}
\newtheorem{mythm}[mydef]{\textbf{Theorem}}
\newtheorem*{nothm}{\textbf{Theorem}}
\newtheorem*{nothma}{\textbf{Theorem A}}
\newtheorem*{nothmb}{\textbf{Theorem B}}
\newtheorem*{nothmc}{\textbf{Theorem C}}
\newtheorem*{nothmd}{\textbf{Theorem D}}
\newtheorem*{nothme}{\textbf{Theorem E}}
\newtheorem*{nothmf}{\textbf{Theorem F}}
\newtheorem*{nothmg}{\textbf{Theorem G}}
\newtheorem{mytheorem}[mydef]{\textbf{Theorem}}

\newtheorem{lem}[mydef]{\textbf{Lemma}}
\newtheorem{pro}[mydef]{\textbf{Proposition}}

\newtheorem{claim}[mydef]{\textbf{Claim}}
\newtheorem{cor}[mydef]{\textbf{Corollary}}

\section{Introduction}

\subsection{Overview}

A hypergroup generalizes the classical notion of a group in such a way that one allows multiplication of a group to be \emph{set-valued}. For instance, for the two point set $\{0,1\}$, one may define addition as $1+1=0$, $1+0=1$, and $0+0=0$. Then one obtains the abelian group with two elements. If one changes the first addition to be $1+1=1$, then it becomes the \emph{Boolean semigroup}. Finally, if one assumes the first addition to be $1+1=\{0,1\}$ then it becomes a \emph{hypergroup}.\\
Contrary to its seemingly unnatural and exotic definition, hypergroups (in general algebraic structures which allow `multi-valued' operations) appear very naturally. For example, let $G$ be a finite group and $H$ be the set of all irreducible (finite dimensional) representations of $G$ (up to equivalence). Then for any $x,y \in H$, the tensor product $x \otimes y$ uniquely decomposes into a direct sum of elements in $H$. Therefore, the tensor product $\otimes$ provides a multi-valued operation on $H$. Of course, a hypergroup is not just a set with a multi-valued binary operation; a hypergroup satisfies certain `group-like' axioms such as the existence of a unique identity and inverses (see $\S \ref{abs}$ for the precise definition).\\ 
In fact, hypergroups emerge in many fields of mathematics, for instance, group theory, algebraic combinatorics, harmonic analysis, and coding theory to name a few.  Also, recently hyperrings and hyperfields (similar generalizations of rings and fields) receive intensive attention due to their various applications. For example, hyperrings and hyperfields are studied in connection to tropical geometry by Viro in \cite{viro}, in connection to number theory by Connes and Consani in \cite{con3}. Also, in \cite{buchstaber1997multivalued}, Buchstaber and Rees introduced a notion of an \emph{$n$-Hopf algebra} by generalizing the classical notion of a Hopf algebra and proved that the rings of functions on a certain class of hypergroups ($n$-valued groups) have $n$-Hopf algebra structures. This can be viewed as a partial generalization of the classical result that the category of Hopf algebras is equivalent to the opposite category of the category of affine group schemes. Moreover, in the recent paper \cite{baker2016matroids}, Baker and Bowler unified various generalizations of matroids (oriented, valuated, and phased) by means of hyperfields.\\

Associations schemes are also generalizations of groups and have been intensively studied in algebraic combinatorics. In \cite{zieschangmax}, Zieschang first made a connection between association schemes and hypergroups. The idea is that any association scheme has a canonical hypergroup structure by means of its complex multiplication. Zieschang also elegantly recast a notion of Tits buildings by using association schemes (see \cite{zieschang1996algebraic} or \cite{zieschangmax}).\\ 
The category of association schemes has been introduced only recently by Hanaki in \cite{hanaki2010category}. But, the category defined by Hanaki had several undesirable properties, for example the image of a morphism is not a sub-association scheme in general. In \cite{french2012functors}, French fixed these problems by implementing a new class of morphisms called admissible morphisms and reinterpreted the group correspondence (of Zieschang) as adjunctions of certain categories.\\ 

In this paper, we study Zieschang's observation in details and make further connections among three objects; association schemes, hypergroups, and projective geometries. To be precise, let $\mathcal{P}$ be a Desarguesian projective geometry of dimension $\geq 2$ which is equipped with a two-sided incidence group $G$. Suppose further that each line $\mathcal{L}$ of $\mathcal{P}$ contains at least four points. Then we canonically attach an association scheme to $\mathcal{P}$ in such a way that complex product encodes incidence relations. To be precise, the following is our initial motivation. 

\subsection{Motivation}\label{motivation}
The initial motivation of this paper came from the following two results:
\begin{mythm}$($Zieschang, \cite{zieschangmax}$)$\label{asshyp}
Each association scheme is a hypergroup with respect to the hyperoperation induced by its complex multiplication. 
\end{mythm}
\begin{mythm}$($Connes and Consani, \cite{con3}$)$\label{corr}
There exists one-to-one correspondence between projective geometries (with an extra condition that each line contains at least four points) and a certain class of commutative hypergroups.
\end{mythm}
To be a bit more precise, each association scheme $S$ on a nonempty set $X$ is equipped with a binary operation (complex multiplication). Zieschang proved that $S$ with complex multiplication is a hypergroup with the identity $1_X$.\\ 
On the other hand, to a projective geometry $\mathcal{P}$, one can associate a hypergroup $H$ in such a way that the underlying set of $H$ is the set of points of $\mathcal{P}$ with one extra point and roughly speaking, the hyperoperation is defined by letting $x+y$ be the set of all points in the line which contains $x$ and $y$. For the precise statement, we refer the readers to \cite{con3} or \cite{thas2014hyperfield}.\\  
Based on these two results, a natural question to ask is whether there is a canonical way to attach an association scheme to projective geometry by means of hypergroups. In fact, one can canonically construct an association scheme from projective geometry by using flags (see Example \ref{fanoexample}). However, this construction is not satisfactory in the sense that it is not compatible with the correspondence given in Theorem \ref{corr}. Hence our question is as follows:
\begin{que}\label{question}
Let $H$ be a commutative hypergroup which corresponds to a projective geometry $\mathcal{P}$ with an extra condition on the number of points on lines. Can we always realize $\mathcal{P}$ as an association scheme $S$ (on a set $X$) such that the following diagram commutes:
\begin{equation}
\begin{tikzcd}[column sep=huge]
(X,S)\arrow{r}{\textrm{Theorem \eqref{asshyp}}} \arrow{d} &\quad H \arrow{dl}{\textrm{Theorem \eqref{corr}}} \\
\mathcal{P} 
\end{tikzcd}
\end{equation}
\end{que}

We prove that the answer is affirmative when $\mathcal{P}$ is equipped with a two-sided incidence group $G$ and satisfies the following condition (see \S\ref{quoquo}):
\begin{itemize}\label{conditions}
\item 
Each line $\mathcal{L}$ of $\mathcal{P}$ contains at least four points. 
\item
$\mathcal{P}$ is Desarguesian and of dimension $\geq 2$. 
\end{itemize}

In particular, this implies that a projective geometry $\mathcal{P}$ with a two-sided incidence group $G$ satisfying the above conditions itself can be considered as an association scheme $S$ and complex multiplication of $S$ encodes incidence information (see Corollary \ref{projpro}). We refer the reader to \cite{con3} for the definition of a two-sided incidence group. 

\subsection{Statement of results}
 This paper is organized as follows.\\ 
In $\S 2$, we review the definitions of association schemes and hypergroups as well as some of their basic properties which will be used in the sequel.\\
In $\S 3$, by adopting a notion of admissible morphisms from \cite{french2012functors}, we construct a functor from the category $\asschmfinite$ of finite association schemes to hypergroups.\\
In $\S 4$, we define that a hypergroup $H$ is \emph{realizable} if $H$ arises from an association scheme and $H$ is \emph{finitely realizable} if $H$ arises from a finite association scheme. We study several properties of realizable hypergroups. In particular, we prove the following. 
\begin{nothma}$($Lemma \ref{product}, Proposition \ref{iff}, Corollary \ref{realcongruence}$)$
\begin{enumerate}
\item 
Let $H_1$ and $H_2$ be realizable hypergroups. Then $H_1\times H_2$ is also realizable.
\item
Sub-hypergroups and quotients of finitely realizable hypergroups are finitely realizable. 
\item
Let $H$ be a hypergroup. Then for any congruence relation (properly defined) $\equiv$ on $H$, $H/\equiv$ is a finitely realizable hypergroup. 
\end{enumerate}
\end{nothma}

In $\S 5$, we provide two interesting classes of realizable hypergroups. The first class naturally arises from a group $G$ and a subgroup $P$ of the group $\Aut(G)$ of automorphisms of $G$ and generalizes the group correspondence of association schemes. As a special case, we prove that projective geometries satisfying the aforementioned conditions are association schemes. Note that this association scheme structure differs from the one obtained from flags (see, Example \ref{fanoexample}). To be precise, we prove the following. 

\begin{nothmb}$($ \S \ref{partition} and \S \ref{quoquo}  $)$\label{thmb}\\
Let $G$ be a group and $P$ be a subgroup of the group $\Aut(G)$ of automorphisms of $G$.
\begin{enumerate}
\item
 We construct an association scheme $(G,P)$ depending on $G$ and $P$. In particular, when $P$ is a trivial group, this recovers the usual association scheme attached to a group $G$.
\item
Let $A$ be a commutative ring and $G$ be a subgroup of the group $A^\times$ of multiplicative units of $A$. For a quotient hyperring $A/G$, the hypergroup $(A/G,\oplus)$ is realizable. 
\end{enumerate}
\end{nothmb}
For the relations between the above theorem and projective geometries, see $\S \ref{quoquo}$.\\  

The last class is obtained from a totally ordered abelian group $(\Gamma,+)$. In this case, a classical result tells us that there exist a field $F$ and a non-Archimedean valuation $\nu$ on $F$ such that $(\Gamma,+)$ is the value group of $\nu$. We say that a triple $(\Gamma,\nu,F)$ satisfies the triangle condition if the following holds:\\

\textbf{(Triangle Condition)} For all $a,b,a',b' \in F$, $r \in \Gamma'=\Gamma \cup \{\infty\}$ such that $\nu(a-b)=\nu(a'-b')=r$, we have the following set bijection of nonempty sets:
\begin{equation}
\{y \in F \mid \nu(a-y)=\nu(y-b)=r\} \simeq \{y \in F \mid \nu(a'-y)=\nu(y-b')=r\}.
\end{equation}
\\
We first review how one enriches a totally ordered abelian group $(\Gamma,\cdot)$ to a hyperfield $(K\Gamma,+,\cdot)$ and prove the following. 

\begin{nothmc}$($\S \ref{linear}$)$
\item
Let $\Gamma$ be a totally ordered abelian group which satisfies the triangle condition for some non-Archimedean valuation $\nu$. Then the hypergroup $(K\Gamma,+)$ is realizable. 
\end{nothmc}

Note that linearly ordered hypergroups (or linearly ordered hyperfields) have been mainly used in tropical geometry for algebraic foundation of tropical geometry. See, \cite{viro} or \cite{jun2015algebraic} for more details. \\

\textbf{Acknowledgment}
\vspace{0.1cm}

This paper was initiated by the conversation with Paul-Hermann Zieschang. The author thanks to him for explaining how one can attach an association scheme to projective geometry by using flags and also for helpful answers for the author's questions on association schemes. The author also thanks to Christopher French for very helpful conversations which helped the author greatly to shape the paper as well as various comments on the first draft. Finally, the author thanks to the referee for useful comments. 

\section{Basic Definitions and Properties}
\subsection{Association schemes}
In this subsection, we recall the definition of association schemes and provide some examples. 
\begin{mydef}
Let $X$ be a nonempty set.
\begin{enumerate}
\item
$1_X$ is the diagonal of $X\times X$; $1_X:=\{(x,x)\mid x \in X\}$.
\item
For each subset $p$ of $X\times X$, we define $p^*:=\{(a,b)\mid (b,a) \in p\}$.
\item
For $x \in X$ and $p \subseteq X \times X$, we define $xp:=\{y \in X \mid (x,y) \in p\}$.
\end{enumerate}
\end{mydef} 

\begin{mydef}
Let $X$ be a nonempty set. We fix a partition $S$ of $X\times X$ and assume that $1_X \in S$. Furthermore, we assume that if $p \in S$, then $p^* \in S$. The set $S$ is an \textbf{association scheme} on a set $X$ if $S$ satisfies the following additional condition: $\forall p,q,r \in S$,
\begin{equation}\label{scheme}
\exists \textrm{ a cardinal number }a_{pq}^r \textrm{ such that } |yp \cap zq^*|=a_{pq}^r \quad \forall y \in X, z\in yr.
\end{equation}
For any three elements $p,q$, and $r$ in $S$, the cardinal number $a_{pq}^r$ is called the structure constant defined by $p,q,$ and $r$. When the underlying set $X$ is finite, we call $(X,S)$ a \textbf{finite association scheme}. 
\end{mydef}
\begin{rmk}
In general, the structure constants $\{a_{pq}^r\}$ of an association scheme $(X,S)$ do not have to be finite. However, when $(X,S)$ is a finite association scheme, $S$ is necessarily a finite set. Also, since $a_{pq}^r$ counts the number of elements in a set, if $a_{pq}^r \neq 0$, then $1\leq a_{pq}^r$. 
\end{rmk}

Let $S$ be an association scheme (on a set $X$). For nonempty subsets $P$ and $Q$ of $S$, one defines the \textbf{complex multiplication} $PQ$ of $P$ and $Q$ as follows:
\begin{equation}\label{complexmultiplication}
PQ:=\{r \in S \mid \exists p \in P,q \in Q \textrm { such that } 1 \leq a_{pq}^r\}.
\end{equation}
In case when $P=\{p\}$ and $Q=\{q\}$, we denote $PQ:=pq$. In particular, for any $p,q \in S$, we have $pq \neq \emptyset$ (it follows from \cite[Lemma 1.1.3]{zieschang2006theory}).
\begin{mydef}
Let $S$ be an association scheme on a set $X$. $S$ is said to be \textbf{commutative} if $a_{pq}^r=a_{qp}^r$ for all $p,q,r \in S$.
\end{mydef}

An association scheme (equipped with complex multiplication) generalizes a group in the following sense.

\begin{myeg}\label{groupcoress}$($\textbf{Group Correspondence}$)$\\
Let $G$ be a finite group. One can associate an association scheme to $G$ in the following way: Let $X=G$ be the underlying set. The pairs $(a,b),(c,d) \in X \times X$ are in the same cell if $ab^{-1}=cd^{-1}$ as elements of the group $G$. Then we have an association scheme $S=\{[g] \mid g \in G\}$, where $[g]:=\{(a,b) \in X \times X \mid ab^{-1}=g\}$. One can easily check that the structure constants are given by: for $[g],[h],[t] \in S$,
\[
a_{[g][h]}^{[t]} =\left\{ \begin{array}{ll}
1 & \textrm{if $t=gh$},\\
0& \textrm{if $t\neq gh$}.
\end{array} \right.
\] 
In fact, the group correspondence is valid for any group $G$ which is not necessarily finite (see \cite{zieschang2006theory}). We will generalize this construction in $\S \ref{partition}$. 
\end{myeg}
\begin{rmk}
The construction of Example \ref{groupcoress} is called the group correspondence since one can retrieve the original group from the associated association scheme by the thin radical construction. For more details, see \cite{zieschang2006theory} or \cite{french2012functors}.  
\end{rmk}

\begin{myeg}
Let $G$ be a connected, regular graph, i.e., every vertex has the same valency (regular) and there is a path between any two vertices (connected). Let $d$ the diameter of $G$ and $d(v,w)$ be the distance between two vertices $v,w \in V(G)$. Recall that $G$ is distance-regular if for any vertices $x,y \in V(G)$ and $0\leq i,j \leq d$, the cardinality of the following set:
\[
R(x,y,i,j):=\{r \in V(G) \mid d(x,r)=i,\textrm{ }d(y,r)=j\}
\]
only depends on $i,j,$ and $d(x,y)$. One can attach an association scheme to $G$ by defining that the underlying set is the set $V(G)$ of vertices of $G$ and $(v,w),(v',w') \in V(G) \times V(G)$ are in the same cell of a partition if and only if $d(v,w)=d(v',w')$. 
\end{myeg}

The following example is a prototype displaying the correspondence between buildings in the sense of Tits and Coxeter schemes in the sense of \cite[\S 12.3]{zieschang1996algebraic}.

\begin{myeg}$($\textbf{Fano plane}$)$\label{fanoexample}\\
Let $\mathbb{P}^2(\mathbb{F}_2)$ be the projective plane over the field $\mathbb{F}_2$ with two elements (Fano plane). Consider the set $X$ of all flags of $\mathbb{P}^2(\mathbb{F}_2)$. To be precise, we have
\[
X=\{(p,l)\mid p \in \mathbb{P}^2(\mathbb{F}_2),\textrm{ $l$ is a line in $\mathbb{P}^2(\mathbb{F}_2)$ which contains $p$}\}.
\]
We define the following subsets of $X \times X$:
\begin{enumerate}
\item
$R_0:=\{(x,x) \mid x \in X\}$.
\item
$R_1:=\{(x,y) \mid x=(p,l) \textrm{ and } y=(q,l) \textrm{ such that } p\neq q\}$.
\item
$R_2:=\{(x,y) \mid x=(p,l) \textrm{ and } y=(p,l') \textrm{ such that } l\neq l'\}$.
\item
$R_3:=\{(x,y) \mid \exists z \in X \textrm{ such that } (x,z) \in R_1 \textrm{ and } (z,y) \in R_2\}$
\item
$R_4:=\{(x,y) \mid \exists z \in X \textrm{ such that } (x,z) \in R_2 \textrm{ and } (z,y) \in R_1\}$
\item
$R_5:=\{(x,y) \mid \exists z,w \in X \textrm{ such that } (x,z) \in R_1, (z,w) \in R_2, \textrm{ and } (w,y) \in R_1\}$
\end{enumerate}
Then the collection $\mathcal{R}=\{R_0,...,R_5\}$ becomes an association scheme on the set $X$. One can easily check that the association scheme is non-commutative. 
\end{myeg}

Next, we recall the definition of morphisms and admissible morphisms. For these morphisms, we restrict ourselves to finite association schemes. 

\begin{mydef}(\cite[Definition 3.1]{french2012functors})\label{mor}\\
Let $X$ and $Y$ be finite sets. Let $S$ and $T$ be association schemes on $X$ and $Y$ respectively.
\begin{enumerate}
\item
A morphism from $S$ to $T$ is a function $f: X \cup S \longrightarrow Y \cup T$ such that $f(X) \subseteq Y$, $f(S) \subseteq T$ and satisfies the following condition:
\begin{equation}\label{morphcondi}
(f(x),f(y)) \in f(p),\quad \forall (x,y) \in p,\textrm{ } p \in S.
\end{equation}
\item
A morphism $f$ from $S$ to $T$ is admissible if $x \in X$, $y \in Y$, and $p \in S$ such that $(f(x),y) \in f(p)$, then there exists $z \in X$ such that $(x,z) \in p$ and $f(z)=y$. 
\end{enumerate}
\end{mydef}

\begin{rmk}
In \cite{zieschang2006theory}, Zieschang refined the definition of morphisms by introducing homomorphisms; however, in general, the composition of two homomorphisms fails to be a homomorphism and hence one can not define the category of association schemes with homomorphisms. On the other hand, with admissible morphisms, French constructed the category $\asschmfinite$ of finite association schemes. The category $\asschmfinite$ fulfills nice properties, for instance, any morphism in $\asschmfinite$ can be factored into the composition of an injection and a surjection. For more details about the category $\asschmfinite$, we refer the readers to \cite[\S 3]{french2012functors}.
\end{rmk}

\subsection{Hypergroups}\label{abs}
In this subsection, we present the definition of hypergroups and also provide basic examples. We refer the readers to \cite{corsini2003applications} for introduction to hypergroups, in particular, for geometric aspects. \\

Let $H$ be a nonempty set. By a hyperoperation $*$ on $H$ we mean a function 
\[*: H\times H \longrightarrow P^*(H),\] 
where $P^*(H)$ is the set of nonempty subsets of $H$. For any nonempty subsets $X,Y \subseteq H$, we use the notation 
\[X*Y:=\bigcup_{x \in X, y \in Y}x*y\] 
to denote the `product' of two subsets.\\ 
A nonempty set $H$ equipped with a hyperoperation $*$ is called a hypergroup if $(H,*)$ satisfies the following conditions:
\begin{enumerate}
\item
associativity: $(a*b)*c=a*(b*c)$.
\item
identity: $\exists ! e \in H$ such that $e*x=x*e=x$ $\forall x \in H$.
\item
inverse: $\forall f \in H$ $\exists ! g(:=f^{-1}) \in H$ such that $e \in (g*f) \cap (f*g)$.
\item 
reversibility: $c \in a*b \textrm{ implies that } a \in c*b^{-1} \textrm{ and } b \in a^{-1}*c, \quad \forall a,b,c \in H.$
\end{enumerate} 
For a hypergroup $H$, if $a*b=b*a$ for any $a,b \in H$, then $H$ is called a commutative hypergroup. 

\begin{myeg}\label{krasnerasgroup}
Let $\mathbf{K}:=\{0,1\}$ be a two point set. One imposes a commutative hyperoperation $+$ as follows:
\[
0+0=0, \quad 1+0=1, \quad 1+1=\{0,1\}. 
\]
Then $\mathbf{K}$ becomes a hypergroup. 
\end{myeg}
\begin{myeg}\label{signhypergroup}
Let $\mathbf{S}:=\{-1,0,1\}$ be a three point set. One imposes a commutative hyperoperation $+$ on $\mathbf{S}$ following the rule of signs, i.e., 
\[
0+0=0, \quad 1+1=1, \quad (-1)+(-1)=-1, \quad 1+(-1)=\{-1,0,1\}.
\]
Then $\mathbf{S}$ becomes a hypergroup.
\end{myeg}

\begin{myeg}
Let $G$ be a finite group and $P$ be the subgroup of the group $\Aut(G)$ of automorphisms of $G$ which consists of inner automorphisms. Let $X:=\{[g] \mid g \in G\}$ be the set of orbits in $G$ under $P$. One imposes a hyperoperation $*$ on $X$ as follows: for $[a],[b] \in X$,
\[
[a]*[b]:=\{[c] \mid c=g_1(a)g_2(b)\textrm{ for some }g_1,g_2 \in P\}.
\] 
Then $X$ becomes a hypergroup (see \cite{campaigne1940partition}). This example will be considered in more general setting in $\S \ref{partition}$. 
\end{myeg}


\begin{myeg}\label{tropicalhyper}
Let $S:=\mathbb{R}\cup \{\infty\}$ be a totally ordered set, where $\mathbb{R}$ is the set of real numbers with usual order and $\infty$ as the greatest element. One imposes the following commutative hyperoperation $+$ on $S$: for $x,y \in S$,
\[
x+ y =\left\{ \begin{array}{ll}
\min\{x,y\} & \textrm{if $x\neq y$},\\
\left[x,\infty\right]& \textrm{if $x=y$}.
\end{array} \right.
\]
Then $(S,+)$ becomes a hypergroup (see \cite{viro}). This example will be considered in $\S \ref{linear}$. 
\end{myeg}

We show that the above examples can be realized as association schemes in \S \ref{examples} in more general setting. Next, we recall the definition of homomorphism between hypergroups.

\begin{mydef}
A \textbf{homomorphism} of hypergroups $(H_1,*_1)$, $(H_2,*_2)$ is a function $f:H_1 \to H_2$ such that 
\[f(a*_1b) \subseteq f(a)*_2f(b) \textrm{ for all } a,b \in H_1.\]
We call $f$ is \textbf{strict} if 
\[f(a*_1b) = f(a)*_2f(b) \textrm{ for all } a,b \in H_1.\]
We let $\hypergroup$ be the category of hypergroups and $\hypergroupfinite$ be the category of finite hypergroups. 
\end{mydef}
\begin{rmk}\label{quo}
One can easily obtain hypergroups from the classical objects via `quotient' construction. See $\S \ref{quoquo}$ for details. We also note that Connes and Consani consider this construction as the scalar extension functor from the category of commutative rings to the category of hyperrings (see,\cite{con4}).
\end{rmk}

\begin{mydef}
Let $H$ be a hypergroup. A hypergroup $L$ is a hypergroup extension of $H$ if there exists an injective hypergroup homomorphism from $H$ to $L$. Also, a subset $K$ of $H$ is called a sub-hypergroup if it is a hypergroup with respect to the hyperoperation which one obtains by restricting the domain of the hyperoperation of $H$ to $K\times K$ and its codomain to $P^*(K)$, the set of nonempty subsets of $K$. 
\end{mydef}

\begin{myeg}
The hypergroup in Example \ref{tropicalhyper} is a hypergroup extension of the hypergroup $\mathbf{K}=\{0,1\}$ in Example \ref{krasnerasgroup} via the injection sending $0$ to $\infty$ and $1$ to $1$. 
\end{myeg}

\section{A Functor from Associations schemes to Hypergroups} \label{functor}
In what follows for each association scheme $S$, we let $\mathbf{H}(S)$ be the hypergroup associated to $S$ as in \cite{zieschangmax}. Recall that $\asschmfinite$ is the category of finite association schemes with admissible morphisms and $\hypergroupfinite$ is the category of finite hypergroups. The following proposition shows that $\mathbf{H}(-)$ is indeed a functor from $\asschmfinite$ to $\hypergroupfinite$. All sets are assumed to be nonempty unless otherwise stated. 

\begin{pro}\label{fun}
Let $X$ and $Y$ be finite sets and $S$ and $T$ be association schemes on $X$ and $Y$ respectively. Let $\varphi:S\longrightarrow T$ be an admissible morphism of association schemes from $S$ to $T$ and $\mathbf{H}(\varphi):S\longrightarrow T$ be the restriction of $\varphi$ to $S$. Then
\[\mathbf{H}(\varphi):\mathbf{H}(S) \longrightarrow \mathbf{H}(T)\] 
is a homomorphism of hypergroups. Furthermore, for $\varphi\circ \psi$, we have $\mathbf{H}(\varphi\circ \psi)=\mathbf{H}(\varphi)\circ \mathbf{H}(\psi)$.
\end{pro}
\begin{proof}
Let $\star$ be complex multiplications of association schemes. Since $\mathbf{H}(\varphi)$ maps $1_X$ to $1_Y$ we only have to show that for any $p,q \in S$, $\varphi(p\star q)\subseteq \varphi(p)\star \varphi(q)$. Let $r \in p\star q$. This means that for any $(x,z) \in r$,
\begin{equation}\label{constant}
a_{pq}^r=|\{y \in X \mid (x,y)\in p, (y,z)\in q\}| \neq 0.
\end{equation}
Therefore, for each $(x,z) \in r$, there exist at least one $y \in X$ such that $(x,y) \in p$ and $(y,z) \in q$. Now, we want to show that $\varphi(r) \in \varphi(p)\star \varphi(q)$. But, since $(\varphi(x), \varphi(z)) \in \varphi(r)$, for $y$ as in \eqref{constant}, we have
\[
(\varphi(x),\varphi(y)) \in \varphi(p), \quad (\varphi(y),\varphi(z))\in \varphi(q).
\]
The second assertion is clear since $\mathbf{H}(\varphi)$ is just a restriction of $\varphi$.
\end{proof}

\begin{cor}
$\mathbf{H}(-):\asschmfinite \longrightarrow \hypergroupfinite$ is a functor.
\end{cor}
\begin{proof}
It directly follows from Proposition \ref{fun}.
\end{proof}

\begin{rmk}
In Proposition \ref{fun}, we did not use the property of admissible morphisms. Therefore, the same statement is true when $\varphi$ is just a morphism (see, Definition \ref{mor}).
\end{rmk}

Let's recall some definitions.

\begin{mydef}$($cf. \cite{zieschang2006theory} for association schemes and \cite{zieschangmax} for hypergroups$)$\label{normalcondition}\\
Let $S$ be an association scheme on a set $X$ and $H$ be a hypergroup.  
\begin{enumerate}
\item
A closed subset $T$ of an association scheme $(X,S)$ is a subset $T$ of $S$ such that $T^*T \subseteq T$. $T$ is normal if $sT=Ts$ for all $s \in S$. $T$ is strongly normal if $T=s^*Ts$ for all $s \in S$. 
\item
A sub-hypergroup $L$ of $H$ is normal if $hL=Lh$ for all $h \in H$. $L$ is strongly normal if $L=h^{-1}Lh$ for all $h \in H$. 
\end{enumerate}
\end{mydef}

\begin{rmk}\label{quoremark}
With closed subsets and normal sub-hypergroups one can generalize the quotient construction of groups. For details, see the aforementioned references. 
\end{rmk}

\begin{pro}\label{normalproposition}
Let $S$ be an association scheme on a set $X$ and $T$ be a closed subset of $S$. 
\begin{enumerate}
\item
$\mathbf{H}(T)$ is a sub-hypergroup of $\mathbf{H}(S)$. 
\item 
If $T$ is (strongly) normal, then $\mathbf{H}(T)$ is a (strongly) normal sub-hypergroup of $\mathbf{H}(S)$.
\end{enumerate}
\end{pro}
\begin{proof}
$(1)$ directly follows from the fact that if $T$ is a closed subset then $1_X \in T$ and $t^* \in T$ for all $t \in T$ (see \cite[\S 2]{zieschang2006theory}). $(2)$ is clear from the definition. 
\end{proof}

\section{Realizable hypergroups}
In the first subsection, we investigate basic properties of hypergroups which arise from association schemes. In the second subsection, we define a notion of congruence relations for hypergroups and then we study how this can be related to realizable hypergroups. 
\subsection{Basic properties of realizable hypergroups}
In this section, we investigate which classes of hypergroups can be seen as association schemes with complex multiplication.
\begin{mydef}
By a realizable hypergroup, we mean a hypergroup $H$ which can be obtained from an associations scheme, i.e., the essential image of the functor $\mathbf{H}$ in \S \ref{functor}. In particular, if a hypergroup $H$ can be obtained from a finite association scheme, we say that $H$ is a finitely realizable hypergroup.  
\end{mydef}



\begin{myeg}
Since hypergroups generalize the notion of groups, it follows from the group correspondence (Example \ref{groupcoress}) that all groups are realizable hypergroups.  
\end{myeg}


In what follows, by $(X,S)$ we will mean an association scheme $S$ on a set $X$ unless otherwise stated. For association schemes $(X_1,S_1)$ and $(X_2,S_2)$ one can define the product (which is again an association scheme on $X_1\times X_2)$ of $(X_1,S_1)$ and $(X_2,S_2)$ (cf. \cite[\S 7]{french2012functors}, \cite{zieschang2006theory}). We use the same notation as in \cite{french2012functors}. In particular, following \cite{french2012functors}, we denote the product of $(X_1,S_1)$ and $(X_2,S_2)$ by $S_1\boxtimes S_2$.

\begin{lem}\label{product}
Let $(X_1,S_1)$ and $(X_2,S_2)$ be association schemes. Suppose that $S_1 \boxtimes S_2$ is the product of $(X_1,S_1)$ and $(X_2,S_2)$. Then the following canonical map:
\[\mathbf{H}(S_1)\times \mathbf{H}(S_2) \longrightarrow \mathbf{H}(S_1\boxtimes S_2), \quad (a,b) \mapsto \zeta_{X_1,X_2}(a,b):=[a,b]
\]
is a strict homomorphism of hypergroups.
\end{lem}
\begin{proof}
Let $\alpha=[p_1,p_2]$ and $\beta=[q_1,q_2]$. We should show that $\alpha * \beta =[p_1*q_2,p_2*q_2]$; however, it directly follows from \cite[Theorem 7.2.3]{zieschang2006theory}. In fact, suppose that $\gamma=[r_1,r_2] \in \alpha *\beta$. This is equivalent to $a_{\alpha\beta}^{\gamma} >0$. However, it follows from \cite[Theorem 7.2.3]{zieschang2006theory} that
\[a_{\alpha\beta}^\gamma = a_{p_1q_1}^{r_1}a_{p_2q_2}^{r_2}.\]
Therefore, $\gamma=[r_1,r_2] \in \alpha *\beta$ is equivalent to $r_1 \in p_1*q_1$ and $r_2 \in p_2*q_2$ and hence we obtain the desired result.
\end{proof}

\begin{pro} \label{productrealiz}
Let $H_1$,..., $H_n$ be realizable hypergroups. Then $H_1\times\cdots \times H_n$ is realizable.
\end{pro}
\begin{proof}
It is enough to show when $n=2$. Suppose that $(X_1,S_1(H_1))$ and $(X_2,S_2(H_2))$ are associations schemes which realize $H_1$ and $H_2$ respectively. Let $S_i(H_i):=S_i$ for $i=1,2$. We claim that $H_1 \times H_2$ is isomorphic to $\mathbf{H}(S_1\boxtimes S_2)$; this will show that $H_1\times H_2$ is realizable since $S_1 \boxtimes S_2$ is an association scheme on $X_1\times X_2$. Define the following map as in Lemma \ref{product}:
\[\varphi: \mathbf{H}(S_1)\times \mathbf{H}(S_2) \longrightarrow \mathbf{H}(S_1\boxtimes S_2), \quad (a,b) \mapsto \zeta_{X_1,X_2}(a,b):=[a,b].
\]
It follows from Lemma \ref{product} that $\varphi$ is strict and also $\varphi$ is clearly surjective. Finally, from \cite[Lemma 7.2.1]{zieschang2006theory}, $\Ker\varphi=\{1_{X_1\times X_2}\}$. Therefore, it follows from the first isomorphism theorem of hypergroups (cf. \cite{corsini2003applications}, \cite{Dav2}) that $\varphi$ is an isomorphism. 
 \end{proof}

\begin{rmk}
The proof of Proposition \ref{productrealiz} actually shows that if hypergroups $H_1,...,H_n$ are finitely realizable then so is their product $H_1\times\cdots \times H_n$. 
\end{rmk}

\begin{rmk}
The author learned from French that the Hamming association schemes $H(n,2)$ provide examples of realizable hypergroups which are not groups. In particular, for each $n \in \mathbb{N}$, there exists a realizable hypergroup of order $n$ which is not a group. 
\end{rmk}


Next proposition shows that being realizable is stable under taking quotients. For the definitions and properties of quotients of association schemes, we refer the readers to \cite{zieschang2006theory} or \cite[\S 2]{french2012functors}. For hypergroups, see \cite{zieschangmax}. Again we use the same notation as in \cite{french2012functors} for association schemes. 

\begin{pro}\label{quotient}
Let $H$ be a finitely realizable hypergroup and $N$ be a normal sub-hypergroup of $H$. Then, $H/N$ is finitely realizable.
\end{pro}
\begin{proof}
Let $S$ be an association scheme on a finite set $X$ such that $H=\mathbf{H}(S)$. Then we can consider $N$ as a normal closed subset of $S$ (cf. Proposition \ref{normalproposition}). We claim that \[H/N \simeq\mathbf{H}(S\sslash N),\] 
where $H/N$ is the quotient hypergroup and $S\sslash N$ is the quotient association scheme (of $S$ by $N$). Consider the following map:
\[\varphi: \mathbf{H}(S\sslash N) \longrightarrow H/N, \quad a^N \mapsto aN.\]
Since $N$ is normal, it follows from \cite[Lemma 4.1.1]{zieschang2006theory} that
\[a^N=b^N \iff NaN=NbN \iff aN=bN.\]
Therefore, $\varphi$ is bijective. Furthermore, $\varphi$ is a strict homomorphism of hypergroups. In fact, it follows from \cite[Theorem 4.1.3 (ii)]{zieschang2006theory} that
\[(a^{r^N}_{p^Nq^N})n_N=\sum_{u \in NpN}\sum_{v \in NqN}a^r_{uv},\]
where $n_N$ is the valency of $N$, i.e., $n_N=\sum_{r \in N} n_r$ with $n_r=a_{rr^*}^{1_X}$. However, since $n_N >0$ and $NpN=pN$ and $NqN=qN$, we have
\[(a^{r^N}_{p^Nq^N})>0 \iff \sum_{u \in pN}\sum_{v \in qN}a^r_{uv} >0.\]
Thus,
\[r^N \in p^Nq^N \iff (a^{r^N}_{p^Nq^N})>0 \iff a^r_{uv}>0 \textrm{ for some }u \in pN, v \in qN.\]
This, in turn, is equivalent to $rN \in (pN)(qN)$. This proves that $\varphi$ is a strict homomorphism of hypergroups which is also bijective and hence $\varphi$ is an isomorphism.
\end{proof}

\begin{pro}\label{sub}
Let $H$ be a realizable hypergroup. Then any sub-hypergroup $T$ of $H$ is also realizable. 
\end{pro}
\begin{proof}
Let $(X,S)$ be an association scheme which realizes $H$ and identify $S$ with $H$. Since $T$ is a sub-hypergroup of $H$, $T$ is a closed subset of $S$ as well. Fix $x_0 \in X$ and we define the following set:
\[
Y:=\{y \in X \mid (x_0,y) \in T\}. 
\]
For each $s \in S$, we let $s_Y:=s \cap (Y\times Y)$. It is known that $T_Y=\{t_Y \mid t \in T\}$ is an association scheme (see, \cite[Theorem 2.1.8]{zieschang2006theory}). We claim that $(Y,T_Y)$ realizes $T$. Let $\star$ be complex multiplications of association schemes. We show that the following map
\[
\varphi: T \longrightarrow T_Y, \quad t \mapsto t_Y
\]
is a strict homomorphism of hypergroups which is bijective. Indeed, clearly $\varphi$ is surjective by definition. Now, suppose that $s_Y=t_Y$ for $s,t \in T$. This implies that we have $(y,z) \in s_Y=t_Y$, or $(x_0,y) \in R_1$ and $(x_0,z) \in R_2$ for some $R_1,R_2 \in T$. Since $T^*T \subseteq T$ and $T$ is a subset of a partition $S$ of $X\times X$, it follows that $(y,z) \in R$ for a unique $R \in T$. This implies that $R=s=t$ and hence $\varphi$ is injective as well. All it remains to show is that $\varphi$ is strict. In other words, we have to show that 
\[
(s\star t)_Y=(s_Y)\star (t_Y). 
\]
However, this directly follows from \cite[Theorem 2.1.8]{zieschang2006theory} which states that
\[
a_{p_Yq_Y}^{r_Y}=a_{pq}^r. 
\] 
\end{proof}

\begin{pro}\label{iff}
Let $B$ be a finitely realizable hypergroup Then any normal sub-hypergroup $A$ and the quotient hypergroup $B/A$ are finitely realizable. 
\end{pro}
\begin{proof}
It follows from Propositions \ref{quotient} and \ref{sub} that if $B$ is realizable then $A$ and $B/A$ are realizable. 
\end{proof}

\subsection{Congruence relations on hypergroups}
In this subsection we consider congruence relations on hypergroups. The case when hypergroups are commutative was considered in \cite{jun2015algebraic}. In general, a congruence relation is an equivalence relation which preserves algebraic structures. But, since a hyperoperation of two elements is not an element but a set, one needs to define the meaning of two sets are equivalent. The following definition has been introduced in \cite{jun2015algebraic} (for commutative hypergroups).   

\begin{mydef}
Let $H$ be a hypergroup and $\equiv$ be an equivalence relation on $H$. For subsets $A,B \subseteq H$, we say that $A$ is equivalent to $B$ (under $\equiv$) if for any $a \in A$ and $b \in B$, there exist $b' \in A$ and $a' \in B$ such that $a\equiv a'$ and $b \equiv b'$. In this case, we write $A \equiv B$. 
\end{mydef}

\begin{mydef}
Let $H$ be a hypergroup. A congruence relation on $H$ is an equivalence relation $\equiv$ on $H$ such that
\begin{enumerate}
\item
For $a,b,x,y \in H$, if $a \equiv x$ and $b \equiv y$ then $ab \equiv xy$ and $ba \equiv yx$. 
\item \label{redundant}
For $a,b \in H$, if $a \equiv b$ then $a^{-1} \equiv b^{-1}$. 
\end{enumerate}
\end{mydef}

\begin{rmk}
When $H$ is a group, the condition \eqref{redundant} is redundant. For hypergroups, since $a*a^{-1}$ is not in general $\{e\}$ but only required to contain $\{e\}$, the condition \eqref{redundant} is necessary.    
\end{rmk}

Let $H$ be a hypergroup and $\equiv$ be a congruence relation on $H$. Let $(H/\equiv):=\{[a]\mid a \in H\}$ be the set of equivalence classes. Let $*$ be a hyperoperation of $H$. We impose the following hyperoperation $\boxdot$ on $H/\equiv$:
\begin{equation}\label{congruenceop}
[x]\boxdot[y]:=\{[z] \mid z \in x'*y' \textrm{ for some }x',y' \in H \textrm{ such that } [x]=[x'],[y]=[y']\}.
\end{equation}

\begin{pro}
Let $H$ be a hypergroup and $\equiv$ be a congruence relation on $H$. Then $H/\equiv$ is a hypergroup with the hyperoperation \eqref{congruenceop}. 
\end{pro}
\begin{proof}
We use the notation $*$ for a hyperoperation of $H$ and $\boxdot$ for $H/\equiv$. Let $e$ be the identity element of $H$. We claim that $[e]$ is the identity element of $H/\equiv$. Clearly we have $[a] \in [e]\boxdot [a]$ $\forall [a] \in H/\equiv$. Suppose that $[b] \in [e] \boxdot [a]$. This implies that there exist $b',e',a' \in H$ such that $[b]=[b']$, $[e]=[e']$, $[a]=[a']$, and $b' \in e'*a'$. Since $e\equiv e'$ and $e$ is the identity, we have $e*a'=a' \equiv e'*a'$. This means that $b' \equiv a'$ and hence $[a]=[b]$. It follows that $[a]\boxdot [e]=[a]$ and therefore $[e]$ is the identity element of $H/\equiv$. Similarly, we have $[e]\boxdot [a]=[a]$.\\
Next, for any $[a] \in H/\equiv$, one can easily observe that $[e] \in [a]\boxdot [a^{-1}]$. We claim that $[a^{-1}]$ is the unique inverse of $[a]$. Indeed, suppose that $[e] \in [a] \boxdot [b]$. Then we have $e',a',b' \in H$ such that $[e]=[e']$, $[a]=[a']$, $[b]=[b']$, and $e' \in a'*b'$. It follows from the reversibility of $H$ that $a' \in e'*(b')^{-1}$. Notice that since $e \equiv e'$, we have $e*(b')^{-1}=(b')^{-1} \equiv e'*(b')^{-1}$ by multiplying $(b')^{-1}$. This implies that any element $c \in e'*(b')^{-1}$ should be congruent to $(b')^{-1}$, in particular, $a' \equiv (b')^{-1}$. It follows from the condition \eqref{redundant} that
\[
a^{-1} \equiv (a')^{-1} \equiv b'
\]
and hence we have $[a^{-1}]=[b']=[b]$. Similarly, one can show that $[e] \in [a^{-1}]\boxdot [a]$ and such $[a^{-1}]$ uniquely exists. This shows that an inverse is unique. The associativity of $H/\equiv$ easily follows from the associativity of $H$.\\
Finally, the reversibility of $H/\equiv$ also directly follows from that of $H$. For example, if $[c] \in [a]\boxdot [b]$, then we have $c',a',b' \in H$ such that $c' \in a'*b'$ and $[c']=[c]$, $[a']=[a]$, and $[b']=[b]$. It follows from the reversibility of $H$ that $a' \in c'*(b')^{-1}$ and from the above argument on inverses, we have $[a] \in [c]\boxdot [b]^{-1}$. This completes our proof.  
\end{proof}

\begin{pro}\label{kernel}
Let $H$ be a hypergroup and $\equiv$ be a congruence relation on $H$. Then a canonical projection
\[
\pi:H \longrightarrow H/\equiv, \quad a \mapsto [a]
\]
is strict. In particular, $H/\Ker(\pi)$ is isomorphic to $H/\equiv$. 
\end{pro}
\begin{proof}
To avoid any notational confusion, we let $\boxdot$ be a hyperoperation of $H/\equiv$ and $*$ be a hyperoperation of $H$. Once we show that $\pi$ is strict, the result follows from the first isomorphism theorem of hypergroups. Clearly, $\pi$ is a homomorphism of hypergroups. We claim that 
\[
\pi(x)\boxdot\pi(y) \subseteq \pi(x*y).
\] 
Suppose that $[z] \in [x]\boxdot [y]=\pi(x)\boxdot \pi(y)$. This means that there exist $x',y' \in H$ such that $[x]=[x']$, $[y]=[y']$, and $z \in x'*y'$. But since $\equiv$ is a congruence relation, we have $x'*y'\equiv x*y$ and hence there exists $z' \in x*y$ such that $z \equiv z'$. It follows that $[z']=[z] \in [x*y]=\pi(x*y)$ and therefore $\pi(x)\boxdot \pi(y) \subseteq \pi(x*y)$. 
 
\end{proof}

\begin{cor}\label{realcongruence}
Let $H$ be a hypergroup. Then $H$ is finitely realizable if and only if for any congruence relation $\equiv$ on $H$, $H/\equiv$ is finitely realizable. 
\end{cor}
\begin{proof}
Let $H$ be a finitely realizable hypergroup. Then for any congruence relation $\equiv$ on $H$, it follows from Proposition \ref{kernel} that $H/\equiv$ is isomorphic to $H/\Ker(\pi)$, where $\pi:H \longrightarrow H/\equiv$ is a canonical projection. But, from Proposition \ref{quotient}, the hypergroup $H/\Ker(\pi)$ is finitely realizable since $H$ is finitely realizable. Conversely, by considering the trivial congruence relation $\equiv$; $x \equiv y \iff x=y$, we obtain $H=H/\equiv$ and hence $H$ is finitely realizable. 
\end{proof}

\section{Examples of realizable hypergroups}\label{examples}
In this subsection, we present two classes of realizable hypergroups. The first example comes from a group $G$ and a subgroup $P$ of the group $\Aut(G)$ of automorphisms of $G$. The second example comes from totally ordered abelian groups. 
\subsection{Partition hypergroups}\label{partition}
The goal of this subsection is to present an example which demonstrates some possibilities of duplication between theory of hypergroups and theory of association schemes.\\ 

We first recall the definition of partition hypergroups. Let $G$ be a group and $P$ be a subgroup of the group $\Aut(G)$ of automorphisms of $G$. One can define an equivalence relation $\sim$ on $G$ as follows: for $x,y \in G$,
\begin{equation}
x \sim y \iff \exists g \in P \textrm{ such that } y=g(x). 
\end{equation} 
Let $\{G\}_P:=\{[x] \mid x \in G\}$ be the set of equivalence classes under $\sim$. One imposes the following hyperoperation $*$ on $\{G\}_P$: for $[x],[y] \in \{G\}_P$,
\begin{equation}\label{hyperoperation}
[x]*[y]:=\{[z] \mid z=x'y' \textrm{ for some } x',y' \in G \textrm{ such that } [x']=[x],[y']=[y]\}.
\end{equation}

It is well known that $\{G\}_P$ equipped with \eqref{hyperoperation} becomes a hypergroup, but we include the proof. Partly, this is because there are many different (and some of them are not equivalent) definitions of hypergroups in literatures and we want to make sure that the statement holds with our setting. 

\begin{rmk}
In \cite{campaigne1940partition}, partition hypergroups are defined in more general setting by using a notion of conjugations. However, since the definition of hypergroups in \cite{campaigne1940partition} is more general than our definition, we do not know whether more general notion of partition hypergroups can be used in our case except the construction we presented above.   
\end{rmk}

\begin{pro}
Let $\{G\}_P$ be as above. Then, with the hyperoperation $*$ in \eqref{hyperoperation}, $\{G\}_P$ is a hypergroup. 
\end{pro}
\begin{proof}
Let $e$ be the identity element of $G$. Then $[e]$ is the identity element of $\{G\}_P$. Indeed, since any automorphism of $G$ should fix $e$, we have that $[e]=\{e\}$. Since $a=ae$ in $G$, we have $[a] \in [a]*[e]$. If $[b] \in [a]*[e]$, then $b=a'e'$ for some $[a']=[a]$ and $[e']=[e]$. Since $[e]=\{e\}$, this implies that $b=a'$ and $[b]=[a]$. Therefore $[a]*[e]=[a]$ and similarly one obtains $[e]*[a]=[a]$. It follows directly that $[e]$ is the unique identity.\\
Next, the existence of an inverse follows from the existence of an inverse in $G$. In other words, for $[a] \in \{G\}_{P}$, we have $[e] \in [a]*[a^{-1}]$, where $a^{-1}$ is the inverse of $a$ in $G$. Suppose that $[e] \in [a]*[b]$. This means that $e=xy$, where $x=g(a)$ and $y=h(b)$ for some $g,h \in P$. It follows that
\[
b=h^{-1}(y)=h^{-1}(x^{-1})=h^{-1}g(a^{-1})=(h^{-1}g)(a^{-1}) 
\]
and hence $[b]=[a^{-1}]$. This proves that the inverse is unique.\\
Finally, we show that $*$ is associative. Let $[a],[b],[c] \in \{G\}_P$ and
\[
X:=\{[d] \mid d=a'b'c' \textrm{ for some } a',b',c' \in G \textrm{ such that} [a']=[a],[b']=[b],[c']=[c]\}.
\] 
We claim that
\[
([a]*[b])*[c]=X=[a]*([b]*[c]).
\] 
Indeed, if $[t] \in ([a]*[b])*[c]$, then $[t] \in [l]*[c]$ for some $[l] \in [a]*[b]$. In other words, $t=l'c'$ and $l=a'b'$ for some $[l']=[l]$, $[c']=[c]$, $[a']=[a]$, and $[b']=[b]$. It follows that $l'=g(l)$ for some $g \in P$ and hence $g(l)=l'=g(a')g(b')$. Therefore we have $t=g(a')g(b')c'$, where $[g(a')]=[a]$, $[g(b')]=[b]$, $[c']=[c]$, and hence we have $[t] \in X$. This shows that $([a]*[b])*[c] \subseteq X$. Conversely, for any $[d] \in X$, we have $d=a'b'c'$, where $[a']=[a],[b']=[b],[c']=[c]$. Since $[a'b'] \in [a][b]$, $[d] \in ([a]*[b])*[c]$. One can similarly prove that $X=[a]*([b]*[c])$. The reversibility easily follows from that of $G$. 
\end{proof}

On the other hand, one can naturally construct an association scheme from a group $G$ and a subgroup $P$ of $\Aut(G)$. To be specific, we define that two pairs $(x,y),(a,b) \in G\times G$ are in the same cell of a partition of $G\times G$ if and only if $xy^{-1}$ and $ab^{-1}$ are in the same element of $\{G\}_P$. In other words, there exists $g \in P$ such that $xy^{-1}=g(ab^{-1})$. It is straightforward that $\{G\}_P$ can be considered as a partition of $G \times G$ in this way. Proposition \ref{associationschemepartitiongroup} shows that this is, in fact, an association scheme. We call this association scheme as a partition association scheme. 
\begin{rmk}
When $P=\{e\}$ is the trivial subgroup of $\Aut(G)$, we obtain $\{G\}_P=G$ and in this case the association scheme which realizes $\{G\}_P$ is same as one obtains from the group correspondence (see Example \ref{groupcoress}). 
\end{rmk}

\begin{pro}\label{associationschemepartitiongroup}
Let $G$ be a group and $P$ be a subgroup of the group $\Aut(G)$ of automorphisms of $G$. Then the partition association scheme $(G,P)$ is indeed an association scheme. 
\end{pro}
\begin{proof}
Let $S:=\{G\}_P$. With an underlying set $G$, $1_G:=\{(x,x) \mid x \in G\}$ is $[e]$ and hence $1_G \in S$. Also, for $[a]=\{(x,y) \in G\times G \mid xy^{-1} \in [a]\}$, we have $[a]^*=[a^{-1}] \in S$.\\ Now, we check the condition of structure constants. For $p=[a],q=[b],r=[c] \in S$, we claim that the structure constant $a_{pq}^r$ is given as follows:
\begin{equation}\label{structureconstant}
a_{pq}^r =\left\{ \begin{array}{ll}
0 & \textrm{if $[c] \notin [a]*[b]$},\\
|\{t \in G \mid abt^{-1} \in [a], t\in[b]\}|& \textrm{if $[c] \in [a]*[b]$.}
\end{array} \right.
\end{equation}
To prove, let's first assume that $[c] \notin [a]*[b]$ and fix $x,z$ such that $xz^{-1} \in [c]$. Suppose that there exists $y \in G$ such that $xy^{-1} \in [a]$ and $yz^{-1} \in [b]$. Then $xy^{-1}=g_1(a)$ and $yz^{-1}=g_2(b)$ for some $g_1,g_2 \in P$. This implies that $(xy^{-1})(yz^{-1})=xz^{-1}=g_1(a)g_2(b)$ and hence $[xz^{-1}]=[c] \in [a]*[b]$. Therefore, in this case we should have $a_{pq}^r=0$.\\
Next, suppose that $[c] \in [a]*[b]$. Let's fix $x,z$ such that $xz^{-1} \in [c]$. We want to show that the cardinality of the following set
\begin{equation}
A(x,z):=\{y \in G \mid xy^{-1} \in [a],\textrm{ }  yz^{-1} \in [b]\}
\end{equation}  
does not depend on a choice of $x$ and $z$. Since $[c] \in [a]*[b]$, there exist $g_1,g_2 \in P$ such that $c=g_1(a)g_2(b)$. Since $xz^{-1} \in [c]$, we know that $xz^{-1}=g(c)$ for some $g \in P$ and hence we may assume that $xz^{-1}=g_1(a)g_2(b)$ for some $g_1,g_2 \in P$. But, since $[a]=[g_1(a)]$ and $[b]=[g_2(b)]$, we have the following:
\[
A(x,z)=\{y \in G \mid xy^{-1} \in [g_1(a)], \textrm{ }yz^{-1} \in [g_2(b)]\}.
\]
Therefore we may further assume that $xz^{-1}=ab$, or $x=abz$. It follows that we have
\[
A(x,z)=\{y \in G \mid abzy^{-1} \in [a],\textrm{ } yz^{-1} \in [b]\}. 
\] 
The set $A(x,z)$ still depends on $z$. Consider the following map:
\[
\varphi:A(x,z) \longrightarrow \{t \in G \mid abt^{-1}\in [a], t\in[b]\}, \quad y \mapsto yz^{-1}. 
\]
One can easily check that $\varphi$ is well-defined and bijective since $z$ is an element of a group $G$ (hence invertible). This shows that $A(x,z)$ indeed does not depend on $x$ and $z$ and the cardinality of $A(x,z)$ is equal to our claimed cardinality. This proves that $(G,P)$ is an association scheme. 
\end{proof}

\begin{rmk}
When $P$ is a trivial group, each equivalence class contains only one element. In this case, we have that, if $[c] \in [a]*[b]$
\[
\{t \in G \mid abt^{-1} \in [a], t\in[b]\}=\{b\}.
\]
Therefore we obtain the following structure constant:
\[
a_{pq}^r =\left\{ \begin{array}{ll}
0 & \textrm{if $[c] \notin [a]*[b]$},\\
1& \textrm{if $[c] \in [a]*[b]$}.
\end{array} \right.
\]
One can see that this is the structure constant which we obtain from the group correspondence (Example \ref{groupcoress}). 
\end{rmk}

\begin{cor}\label{realizalcor}
Let $G$ be a group and $P$ be a subgroup of the group $\Aut(G)$ of automorphisms of $G$. Then the partition hypergroup $\{G\}_P$ is realizable. 
\end{cor}
\begin{proof}
We prove that the hypergroup $\{G\}_P$ is realized by the partition association scheme constructed in Proposition \ref{associationschemepartitiongroup}. Let $\boxdot$ be the complex multiplication of the partition association scheme $(G,P)$ and $*$ be the hyperoperation of the partition hypergroup $\{G\}_P$. We want to show that
\begin{equation}
[a]\boxdot [b] = [a]*[b],\quad  \forall [a],[b] \in \{G\}_P,
\end{equation}
where on the left hand side, $[a]$ and $[b]$ are considered as elements in the partition association scheme, whereas on the right hand side, they are considered as elements of a partition hypergroup $\{G\}_P$. But, $[c] \in [a]\boxdot [b]$ if and only if $a_{[a][b]}^{[c]} > 0$ if and only if $[c] \in [a]*[b]$. This proves that $(S,\boxdot)$ and $(\{G\}_P,*)$ are isomorphic and therefore $\{G\}_P$ is realizable. 
\end{proof}

\begin{rmk}
In our propositions, we did not assume that a group $G$ is finite. 
\end{rmk}

Next, we provide an example to illustrate some (previously unknown) duplication between theories of hypergroups and association schemes.\\

Let $G$ be a finite group and we fix a group $P$ to be $\Inn(G)$, the group of inner automorphisms of $G$. Then we have the following result of Campaigne. 

\begin{mythm}$($\cite[Theorem 9.4]{campaigne1940partition}$)$ \label{camthm}
Let $G$ be a finite group and $P$ be the group of inner automorphisms of $G$. Then $G$ is simple if and only if the partition hypergroup $\{G\}_P$ has no proper sub-hypergroups except the trivial one $\{e\}$. 
\end{mythm}

There is a similar theorem in scheme theory, and, in the remaining part of this section, we will see that the corresponding scheme theoretic theorem in equivalent to Theorem \ref{camthm}. \\




Let $P=\Inn(G)$. One can easily check that the partition association scheme $(G,P)$ is commutative. In fact, this association scheme has been already considered in many literatures. For instance, see \cite[\S 3.1]{bannai1990orthogonal}. In particular, we have the following:

\begin{mythm}\label{schemesimple}
Let $G$ be a finite group and $P$ be the group of inner automorphisms of $G$. Then the association scheme $(G,P)$ is primitive if and only if $G$ is simple. 
\end{mythm} 

One can immediately notice that Theorem \ref{camthm} and Theorem \ref{schemesimple} are the same statement through our discussion. To be specific, suppose that $H$ is a hypergroup which is realized by an association scheme $(X,S)$. Then each closed subset $T$ of $S$ becomes a sub-hypergroup of $H$. This implies that $(X,S)$ is primitive if and only if $H$ does not have a sub-hypergroup but $\{e\}$. It follows that an association scheme $(G,P)$ is primitive if and only if $\{G\}_P$ does not have a nontrivial sub-hypergroup. This proves that Theorems \ref{camthm} and \ref{schemesimple} are the same theorem.  

\subsection{Quotient construction}\label{quoquo}
In this subsection, we construct an example which naturally arises from the previous subsection $\S \ref{partition}$. We first recall some definitions.

\begin{mydef}\label{hyperringdef}
A hyperring is a set $R$ equipped with two operations $\cdot$, $\oplus$ with two distinguished elements $0_R$ and $1_R$ such that $(R,\cdot,1_R)$ is a commutative monoid and $(R,\oplus,0_R)$ is a commutative hypergroup. Furthermore $\cdot$ and $\oplus$ are compatible, i.e., $(a\oplus b)\cdot c=(a\cdot c)\oplus (b\cdot c)$ and $0_R$ is an absorbing element, i.e., $a\cdot 0_R=0_R$ for all $a \in R$. When $(R\backslash\{0\},\cdot,1_R)$ is a group, we say that $R$ is a hyperfield. 
\end{mydef}

\begin{rmk}
Strictly speaking, the definition of hyperrings given in Definition \ref{hyperringdef} should be called Krasner hyperrings. Hyperrings are more general objects, for instance, one does not assume the commutativity.  
\end{rmk}

\begin{rmk}
We note that recently, in \cite{baker2016matroids}, Baker and Bowler implemented a notion of matroids over hyperfields and generalized various classes of matroids (valuated matroids, oriented matroids, and phased matroids) in a very elegant way. 
\end{rmk}

One can easily construct a hyperring from a commutative ring in the following way:\\
Let $A$ be a commutative ring with $1_A$. Let $G$ be a subgroup of the group $A^\times$ of multiplicatively invertible elements of $A$. Then $G$ induces an equivalence relation $\sim$ of $A$ as follows:
\[
a\sim b \iff a=gb\textrm{ for some } g\in G.
\] 
Let $R:=A/G=\{[a] \mid a \in A\}$ be the set of equivalence classes of $\sim$, where $[a]$ is the equivalence class of $a \in A$. One imposes hyperoperation $\oplus$ and usual binary operation $\star$ on $R$ as follows: let $+$ be the addition and $\cdot$ be the multiplication of $A$, the for $[a],[b] \in R$, 
\begin{itemize}
\item
$[a]\oplus [b]:=\{[c]\mid c=g_1\cdot a+g_2\cdot b\textrm{ for some }g_1,g_2 \in G\}$.
\item
$[a]\star [b]:=[a\cdot b]$.
\end{itemize}
\begin{nothm}$($\cite[Proposition 2.6]{con3}$)$\\
With the same notations as above, $(R,\cdot,\oplus)$ is a hyperring. 
\end{nothm}

\begin{myeg}\label{krasnerhyperfield}
Consider the hypergroup $\mathbf{K}=\{0,1\}$ given in Example \ref{krasnerasgroup}. One can impose a multiplication to $\mathbf{K}$ which is same as that of the field $\mathbb{F}_2$ with two elements. Then $\mathbf{K}$ becomes a hyperfield. In fact, one can obtain $\mathbf{K}$ from the above quotient construction by letting $A$ be any field $k$ such that $|k| \geq 3$ and $G=k^\times$. $\mathbf{K}$ is called the \emph{Krasner hyperfield}. 
\end{myeg}

\begin{pro}\label{projpro}
Let $A$ be a commutative ring with identity and $G$ be a (multiplicative) subgroup of the group $A^\times$ of units in $A$. Then $(A/G,\oplus)$ is realizable.  
\end{pro}
\begin{proof}
Let $H:=(A,+)$ be the (additive) abelian group structure of $A$. Then one can consider $G$ as a subgroup of $\Aut(H)$ in such a way that for each $g \in G$, $g:H \longrightarrow H$ sending $a$ to $ga$ (since $A$ is commutative, whether we multiply $g$ to the left or the right does not makes any difference). Then one can easily see that the hyperaddition $\oplus$ is identical to the hyperaddition $*$ defined in \eqref{hyperoperation}. It follows from Corollary \ref{realizalcor} that $(A/G,\oplus)$ is realizable.  
\end{proof}

\begin{myeg}\label{signquoconstruction}
Let $A=\mathbb{Q}$ be the field of rational numbers and $M=\mathbb{Q}_{>0}$ be the (multiplicative) subgroup of $\mathbb{Q}^\times$ which consists of positive rational numbers. Then $M$ induces an equivalence relation $\sim$ on $A$ as follows: for $a, b \in A$
\[
a \sim b \iff \textrm{ either }a=b=0\textrm{ or } ab \neq 0\textrm{ and } ab^{-1} \in M.
\]
Let $R:=\{[a] \mid a \in A\}$ be the set of equivalence class of $A$ under $\sim$. One can impose two operations similar to above: for $[a],[b] \in R$, 
\begin{itemize}
\item\label{hyperaddition}
$[a]\oplus [b]:=\{[c] \mid c=ag_1+bg_2\textrm{ for some }g_1,g_2 \in M\}$. 
\item
$[a]\cdot [b]:=[ab]$.
\end{itemize} 
One can easily observe that $R$ consists three element $[0],[1],[-1]$ and the above hyperoperation follows the rule of signs:
\[
[1]\oplus [1]=[1],\quad [-1]\oplus[-1]=[-1],\quad [1]\oplus[-1]=\{[0],[-1],[1]\}
\]
Also, the multiplicative structure of $R$ is same as that of $\mathbb{F}_3$, the field with three elements. Equipped with these two operations $\mathbf{S}:=\{0,-1,1\}$ becomes a hyperfield called the \emph{hyperfield of signs}. One can easily see that the argument in Proposition \ref{associationschemepartitiongroup} can be used to prove that $\mathbf{S}$ is realizable by using the underlying set $\mathbb{Q}$ (considered as an additive group) and two pairs $(x,y),(a,b) \in \mathbb{Q} \times \mathbb{Q}$ are in the same cell if and only if $(x-y)$ and $(a-b)$ have the same signs. In fact, French proved that $\mathbf{S}$ is realizable, but not finitely realizable. For more general results about quotient hyperrings, we refer the readers to \cite{con3}. 
\end{myeg}

Next, we explain how Proposition \ref{projpro} is linked to our first motivation given in \S \ref{motivation}. In what follows, we always assume that a projective geometry is finite.\\ 

Let $\mathbf{K}$ be the Krasner hyperfield (Example \ref{krasnerhyperfield}). Recall that one says that a commutative hypergroup $E$ is a $\mathbf{K}$-vector space if 
\begin{equation}
x+x=\{0,x\}, \quad \forall x\neq 0 \in E.
\end{equation} 
In particular, one has the following theorem.
\begin{nothm}$($\cite[Propositions 3.1 and 3.5, Theorem 3.8]{con3}$)$\label{cc}
\begin{enumerate}
\item
There exists one-to-one correspondence between finite $\mathbf{K}$-vector spaces and projective geometries (with some additional condition on number of points in a line).
\item
Let $G$ be a two-sided incidence group of a projective geometry $\mathcal{P}$ such that each line of $\mathcal{P}$ contains at least four points. Then there exists a unique hyperfield extension $L$ of $\mathbf{K}$ such that $\mathcal{P}$ is a projective geometry associated to $L$. 
\item
Let $L$ be a hyperfield extension of $\mathbf{K}$ and $\mathcal{P}$ be the projective geometry associated to $L$ (by considering $L$ as a $\mathbf{K}$-vector space). If $\mathcal{P}$ is Desarguesian and of dimension greater than $1$, there exist a unique pair $(F,K)$ of a field $F$ and a subfield $K$ of $F$ such that $L \simeq F/K^\times$.
\end{enumerate}
\end{nothm}

By combining this theorem with Proposition \ref{projpro}, we obtain the following result. 

\begin{cor}\label{projectivegeoreult}
Let $\mathcal{P}$ be a Desarguesian projective geometry of dimension $\geq 2$, equipped with a two-sided incidence group $G$. Suppose further that each line of $\mathcal{P}$ contains at least four points. Then there exists an association scheme $S$ on a nonempty set $X$ such that $\mathbf{H}(S)$ is a hypergroup which corresponds to $\mathcal{P}$ under the construction of Connes and Consani. In particular, in this case, $\mathcal{P}$ naturally becomes an association scheme in such a way that one can recover all points on the line from complex multiplication. 
\end{cor}

\begin{rmk}
\begin{enumerate}
\item
Since any projective geometry of dimension $\geq 3$ is Desarguesian, Theorem \ref{cc} implies that most projective geometries come from the quotient type. We refer the readers to \cite[\S 3]{con3} for details.
\item
The above theorem of Connes and Consani links the classification problem of finite extensions of $\mathbf{K}$ to an abelian case of a long-standing conjecture on the existence of finite non-Desarguesian plane with a simply transitive group of collineations. For the notion of non-Desarguesian plane, see \cite{weibel2007survey}. Also, for the number theoretic interpretation of the conjecture, see \cite{thas2008finite}. 
\item 
As we previously mentioned, it is well-known that projective geometry has an association scheme structures by means of flags (cf. \cite{chakravarti1993three} or Example \ref{fanoexample}). But, this association scheme is not commutative, in particular, the association scheme structure defined by flags does not realize the hypergroup structure we obtain from projective geometry, whereas Corollary \ref{projectivegeoreult} states that in some cases, projective geometry $\mathcal{P}$ itself can be considered as an association scheme $S$ and the complex multiplication of $S$ agrees with the hypergroup structure which we obtain from the incidence relation of $\mathcal{P}$. 
\end{enumerate}
\end{rmk}

\subsection{Linearly ordered hypergroups}\label{linear}
Recall that to a totally ordered set $(G,\leq)$, one can canonically implement a (commutative) hypergroup structure. In fact, we let $G'=G \cup \{\infty\}$ and give an order to $\infty$ such that $g < \infty$ $\forall g \in G$. One can define hyperoperation on $G'$ as follows:
\begin{equation}\label{linearhyper}
x+ y =\left\{ \begin{array}{ll}
\min\{x,y\} & \textrm{if $x\neq y$}\\
\left[x,\infty\right]:=\{g \in G' \mid x \leq g\}& \textrm{if $x=y$}
\end{array} \right.
\end{equation}
Suppose that $G$ is a totally ordered abelian group. Then one can easily see that $G$ becomes a hyperfield with hyperoperation given in \eqref{linearhyper} (see \cite{viro} for more details). It is well-known that for any totally ordered abelian group $G$, there exist a field $F$ and a non-Archimedean valuation $\nu$ on $F$ such that the value group of $\nu$ is $G$. One can naturally extend the valuation $\nu$ to $\nu:F \longrightarrow G'$ by letting $\nu(0)=\infty$. In this subsection, we will assume that a valuation $\nu$ satisfies the following triangle condition:\\

\textbf{(Triangle Condition)} For all $a,b,a',b' \in F$, $r \in G'$ such that $\nu(a-b)=\nu(a'-b')=r$, we have the following set bijection of nonempty sets:
\begin{equation}\label{triangle}
\{y \in F \mid \nu(a-y)=\nu(y-b)=r\} \simeq \{y \in F \mid \nu(a'-y)=\nu(y-b')=r\}.
\end{equation}
\\
\begin{myeg}
For example, when $F$ is $\mathbb{Q}_p$ ($p$-adic numbers), $\mathbb{C}\{\{t\}\}$ (Puiseux series), or $k((G))$ (Mal'cev-Neumann ring) with $k$ is algebraically closed field and $G$ is a divisible subgroup of $\mathbb{R}$, the triangle condition is fulfilled. For more details about these objects, we refer the readers to \cite[\S 2.1]{bernd}.
\end{myeg}

Now, let's assume that there is a non-Archimedean valuation $\nu:F \longrightarrow G'$ which satisfies the triangle condition \eqref{triangle}. We can naturally associate an association scheme which has an underlying set $F$ as follows. First, for each $g \in G'$, we define the following set:
\[
F_g:=\{(a,b) \in F\times F \mid \nu(a-b)=g\}.
\] 

\begin{rmk}
A non-Archimedean valuation $\nu$ induces a metric topology on $F$. One can think of $F_g$ as the set of pairs of distance $g$ for $g \in G$. Since $G$ is totally ordered, for example such as $\mathbb{R}$, this makes sense. 
\end{rmk}

One can easily see that $\{F_g\}_{g \in G'}$ becomes a partition of $F\times F$ since we assume that the value group of $\nu$ is $G$. Note that $F_{\infty}=\{(a,a)\mid a \in F\}$. Also, since $\nu$ is a valuation, one has 
\[
F_g=\{(a,b) \in F\times F \mid \nu(b-a)=g\}.
\] 
All it remains to show that $\{F_g\}_{g \in G'}$ is an association scheme on $F$ is the structure constants condition. For the notational convenience, let $g:=F_g$ for each $g \in G'=G \cup \{\infty\}$. We compute structure constants to show that $\{F_g\}_{g \in G'}$ is an association scheme. Let $p,g,r \in G'$.\\ 


\textbf{Case 1: when $p,q,r \in G'$ are all distinct.} In this case, we have $a^r_{pq}=0$. In fact, for $(a,b) \in r$ and $y \in F$, we have
\begin{equation}\label{nonarc}
r=\nu(b-a)=\nu(b-y+y-a) \geq \min (\nu(b-y),\nu(a-y)).
\end{equation}
If $p=\nu(b-y)$ and $q=\nu(a-y)$, then since we assumed that $p \neq q$ and $\nu$ is non-Archimedean, this implies that $r=p$ or $r=q$. Thus in this case $a_{pq}^r=0$.\\

\textbf{Case 2: $p=q=r$.} In this case, we have a well-defined $a_{pq}^r$ from the triangle condition which we assumed. Moreover, we have $a_{pp}^p \neq 0$ also from the triangle condition.\\

\textbf{Case 3: $r< p=q$.} In this case, it follows from \eqref{nonarc} that $a_{pq}^r=0$.\\

\textbf{Case 4: $r > p=q$.} For $(a,b),(a',b') \in r$, let us define the following sets:
\[
A:=\{y \in F\mid (a,y) \in p, (y,b)\in p\}, \quad B:=\{y' \in F\mid (a',y') \in p, (y',b')\in p\}.
\]
Let $t =a-a'$ and define the following map:
\[
\varphi: A\longrightarrow B, \quad y \mapsto y'=y-t.
\]
Then $\varphi$ is well-defined. In fact, 
\[\nu(y'-a')=\nu(y-t-a')=\nu(y+a'-a-a')=\nu(y-a).\]
Also, we have
\[
\nu(b'-y')=\nu(b'-a'+a'-y')\geq \min(\nu(b'-a'),\nu(a'-y'))=\min(r,p).
\]
Furthermore, since we assumed that $r >p$ and $\nu$ is non-Archimedean, this implies that $\nu(b'-y')=p$. This shows that $y' \in B$. Clearly, $\varphi$ is a bijection and one can similarly show that $\varphi^{-1}$ is also well-defined. Furthermore, in this case, $a_{pp}^r\neq 0$. In fact, since we assumed that $\nu(F)=G'$, there exist $\gamma,y \in F$ such that $\nu(\gamma)=r$ and $\nu(y)=p$. Hence $(0,\gamma) \in r$ and $(0,y) \in p$. It is enough to show that $(y,\gamma) \in p$ or equivalently, $\nu(\gamma-y)=p$. However, we have
\[
\nu(\gamma-y)\geq \min(\nu(\gamma),\nu(y))=\min(r,p).
\]
But, since $r>p$, this implies that $\nu(\gamma-y)=p$. Therefore $a_{pp}^r \neq 0$.\\

\textbf{Case 5: $r=p >q$ or $r=q >p$.} Since these two are symmetric, it is enough to consider the case when $r=p >q$. Suppose that $(a,b) \in r$, $y \in F$ and $(a,y) \in p$, $(y,b) \in q$. Then we have
\[
r=\nu(a-b)=\nu(a-y+y-b) \geq \min (\nu(a-y),\nu(y-b)) =\min(r,q).
\] 
Again, since $r >q $ and $\nu$ is non-Archimedean, this implies that $r=q$ which is a contradiction. Thus, $a_{pq}^r=0$.\\ 

This proves that $\{F_g\}_{g\in G'}$ is an association scheme. One can, in fact, observe that this association scheme realizes a hypergroup we introduced at the beginning of this subsection (see \eqref{linearhyper}). Precisely, for $F_x,F_y$ with $x, y \in G'$, it follows from the above computations that
\[
F_x * F_y:=\{F_z \mid a_{xy}^z \neq 0\}=\left\{ \begin{array}{ll}
F_{\min\{x,y\}} & \textrm{if $x\neq y$}\\
\{F_z \mid z \geq x\}& \textrm{if $x=y$}
\end{array} \right.
\]
Lastly, we remark that a trivial valuation is a special case of non-Archimedean valuation which satisfies the triangle condition. In this case, we recover the fact that the hypergroup $\mathbf{K}$ in Example \ref{krasnerasgroup} is realizable.

\bibliography{ass}

\begin{thebibliography}{Cam40}

\bibitem[Ban90]{bannai1990orthogonal}
Eiichi Bannai.
\newblock Orthogonal polynomials in coding theory and algebraic combinatorics.
\newblock In {\em Orthogonal Polynomials}, pages 25--53. Springer, 1990.

\bibitem[BB16]{baker2016matroids}
Matthew Baker and Nathan Bowler.
\newblock Matroids over hyperfields.
\newblock {\em arXiv preprint arXiv:1601.01204}, 2016.

\bibitem[BR97]{buchstaber1997multivalued}
VM~Buchstaber and EG~Rees.
\newblock Multivalued groups, their representations and hopf algebras.
\newblock {\em Transformation groups}, 2(4):325--349, 1997.

\bibitem[Cam40]{campaigne1940partition}
Howard Campaigne.
\newblock Partition hypergroups.
\newblock {\em American Journal of Mathematics}, pages 599--612, 1940.

\bibitem[CC10]{con4}
Alain Connes and Caterina Consani.
\newblock From monoids to hyperstructures: in search of an absolute arithmetic.
\newblock {\em Casimir Force, Casimir Operators and the Riemann Hypothesis, de
  Gruyter}, pages 147--198, 2010.

\bibitem[CC11]{con3}
Alain Connes and Caterina Consani.
\newblock The hyperring of adele classes.
\newblock {\em Journal of Number Theory}, 131(2):159--194, 2011.

\bibitem[Cha93]{chakravarti1993three}
Indra~Mohan Chakravarti.
\newblock A three-class association scheme on the flags of a finite projective
  plane and a (pbib) design defined by the incidence of the flags and the baer
  subplanes in pg (2, q 2).
\newblock {\em Discrete mathematics}, 120(1):249--252, 1993.

\bibitem[CL03]{corsini2003applications}
Piergiulio Corsini and Violeta Leoreanu.
\newblock {\em Applications of hyperstructure theory}, volume~5.
\newblock Springer, 2003.

\bibitem[DLF07]{Dav2}
B~Davvaz and V~Leoreanu-Fotea.
\newblock Hyperring theory and applications, 2007.

\bibitem[Fre13]{french2012functors}
Christopher French.
\newblock Functors from association schemes.
\newblock {\em Journal of Combinatorial Theory, Series A}, 120(6):1141 -- 1165,
  2013.

\bibitem[Han10]{hanaki2010category}
Akihide Hanaki.
\newblock A category of association schemes.
\newblock {\em Journal of Combinatorial Theory, Series A}, 117(8):1207--1217,
  2010.

\bibitem[Jun15]{jun2015algebraic}
Jaiung Jun.
\newblock Algebraic geometry over hyperrings.
\newblock {\em arXiv preprint arXiv:1512.04837}, 2015.

\bibitem[MS15]{bernd}
Diane Maclagan and Bernd Sturmfels.
\newblock {\em Introduction to tropical geometry}, volume 161.
\newblock American Mathematical Soc., 2015.

\bibitem[Tha14]{thas2014hyperfield}
Koen Thas.
\newblock Hyperfield extensions, characteristic one and the connes-consani
  plane connection.
\newblock {\em arXiv preprint arXiv:1407.0607}, 2014.

\bibitem[TZ08]{thas2008finite}
Koen Thas and Don~B Zagier.
\newblock Finite projective planes, fermat curves, and gaussian periods.
\newblock {\em Journal of the European Mathematical Society}, 10(1):173--190,
  2008.

\bibitem[Vir10]{viro}
Oleg Viro.
\newblock Hyperfields for tropical geometry i. hyperfields and dequantization.
\newblock {\em arXiv preprint arXiv:1006.3034}, 2010.

\bibitem[Wei07]{weibel2007survey}
Charles Weibel.
\newblock Survey of non-desarguesian planes.
\newblock {\em Notices of the AMS}, 54(10):1294--1303, 2007.

\bibitem[Zie96]{zieschang1996algebraic}
Paul-Hermann Zieschang.
\newblock An algebraic approach to association schemes.
\newblock {\em Lecture Notes in Mathematics}, 1996.

\bibitem[Zie06]{zieschang2006theory}
Paul-Hermann Zieschang.
\newblock {\em Theory of association schemes}.
\newblock Springer Science \& Business Media, 2006.

\bibitem[Zie10]{zieschangmax}
Paul-Hermann Zieschang.
\newblock Hypergroups.
\newblock {\em Max-Planck-Institut f{\"u}r Mathematik Preprint Series}, 2010.

\end{thebibliography}
\bibliographystyle{alpha}
\end{document}